\documentclass[11pt,a4paper]{amsart}
\usepackage[british]{babel}

\usepackage{a4wide}
\usepackage{amssymb}
\usepackage{amsmath}
\usepackage{amsthm}
\usepackage{array}
\usepackage{enumitem}
\usepackage{mathrsfs}
\usepackage[hidelinks]{hyperref}
\usepackage{breakurl}
\usepackage[backref=false,isbn=false,doi=true,backend=bibtex,style=alphabetic]{biblatex}
\newtheorem{thm}{Theorem}[section]

\newtheorem{cor}[thm]{Corollary}
\newtheorem{lem}[thm]{Lemma}

\newtheorem{prop}[thm]{Proposition}

\theoremstyle{definition}

\newtheorem{exmpl}[thm]{Example}

\newtheorem{definition}[thm]{Definition}

\renewcommand{\epsilon}{\varepsilon}
\renewcommand{\phi}{\varphi}
\newcommand{\defeq}{\mathrel{\mathop:}=}

\newcommand{\clone}[1]{\mathscr{C}_{#1}} 
\newcommand{\cclone}[1]{\overline{\mathscr{C}_{#1}}}
\DeclareMathOperator{\inte}{int}
\DeclareMathOperator{\Clo}{\mathscr{C}}
\DeclareMathOperator{\cClo}{\overline{\mathscr{C}}}
\DeclareMathOperator{\HSP}{HSP}

\DeclareMathOperator{\HSPfin}{HSP^{fin}}
\DeclareMathOperator{\DHSPfin}{DHSP^{fin}}
\DeclareMathOperator{\HA}{H}
\DeclareMathOperator{\SA}{S}
\DeclareMathOperator{\PA}{P}
\DeclareMathOperator{\Pfin}{P^{fin}}
\DeclareMathOperator{\DA}{D}
\DeclareMathOperator{\id}{id}
\DeclareMathOperator{\Sym}{Sym}
\DeclareMathOperator{\Aut}{Aut}
\DeclareMathOperator{\Hom}{Hom}
\DeclareMathOperator{\pr}{pr}
\DeclareMathOperator{\Ker}{Ker}

\begin{document}

\setlist{noitemsep}

\title{A uniform Birkhoff theorem}

\author{Friedrich Martin Schneider}
\address{Institute of Algebra, TU Dresden, 01062 Dresden, Germany}

\date{\today}

\begin{abstract} Garret Birkhoff's HSP theorem characterizes the classes of models of algebraic theories as those being closed with respect to homomorphic images, subalgebras, and products. In particular, it implies that an algebra $\mathbf{B}$ satisfies all equations that hold in an algebra $\mathbf{A}$ of the same type if and only if $\mathbf{B}$ is a homomorphic image of a subalgebra of a (possibly infinite) direct power of $\mathbf{A}$. The former statement is equivalent to the existence of a natural map sending term functions of the algebra $\mathbf{A}$ to those of $\mathbf{B}$, and it is natural to wonder about continuity properties of this mapping. We show that this map is uniformly continuous if and only if every finitely generated subalgebra of $\mathbf{B}$ is a homomorphic image of a subalgebra of a finite power of $\mathbf{A}$ -- without any additional assumptions concerning the algebras $\mathbf{A}$ and $\mathbf{B}$. Moreover, provided that $\mathbf{A}$ is \emph{almost locally finite} (for instance if $\mathbf{A}$ is locally oligomorphic or locally finite), the considered map is uniformly continuous if and only if it is Cauchy-continuous. In particular, our results extend a recent theorem by Bodirsky and Pinsker beyond the countable $\omega$-categorical setting. \end{abstract}

\maketitle


\section{Introduction}

Garrett Birkhoff's HSP theorem~\cite{birkhoff} ranges among the most fundamental results in algebra and has evoked a tremendous body of theory. It states that a class of algebras of the same is the class of all models of an algebraic theory if and only if it is closed with respect to homomorphic images, subalgebras, and products. In particular, an algebra $\mathbf{B}$ is a model of the algebraic theory generated by an algebra $\mathbf{A}$ of the same type, i.e., $\mathbf{B}$ satisfies all equations that hold in $\mathbf{A}$, if and only if $\mathbf{B}$ is a homomorphic image of a subalgebra of a (possibly infinite) direct power of $\mathbf{A}$. Of course, this provides a constructive method for studying algebraic theories and hence constitutes a useful source of counterexamples. For a more elaborate account on classical applications of Birkhoff's theorem we refer to \cite{burris}.

The present paper's purpose is to prove a topological variant of Birkhoff's theorem. For this to make sense, we need to view Birkhoff's theorem from a slightly different perspective, which is due to \cite{TopologicalBirkhoff}. To explain this, let $\Sigma = (\Sigma_{n})_{n \in \mathbb{N}}$ be a \emph{signature}, that is, a sequence of disjoint sets. Given any set $X$ of variables, we denote by $T_{\Sigma}(X)$ the set of \emph{$\Sigma$-terms} over $X$, i.e., the smallest set $T$ subject to the following conditions: 
\begin{enumerate}
	\item[---] \ $X \subseteq T$.
	\item[---] \ $\sigma t_{1} \ldots t_{n} \in T$ for all $n \in \mathbb{N}$, $\sigma \in \Sigma_{n}$ and $t_{1},\ldots,t_{n} \in T$.
\end{enumerate} Suppose $\mathbf{A}$ to be an \emph{algebra of type $\Sigma$}, that is, a pair consisting of a set $A$ and a family of functions $\sigma^{\mathbf{A}} \colon A^{n} \to A$ for $\sigma \in \Sigma_{n}$, $n \in \mathbb{N}$. Furthermore, let $(x_{n} \mid n \geq 1 )$ be a sequence of pairwise distinct variables. For each $n \geq 1$, we put $X_{n} \defeq \{ x_{1},\ldots,x_{n} \}$ and define a map \begin{displaymath}
	\nu^{\mathbf{A}}_{n} \colon T_{\Sigma}(X_{n}) \to A^{A^{n}}
\end{displaymath} by recursive term evaluation as follows: if $a \in A^{n}$, then $\nu^{\mathbf{A}}_{n}(x_{i})(a) \defeq a_{i}$ for all $i \leq n$, and \begin{displaymath}
	\nu^{\mathbf{A}}_{n}(\sigma t_{1}\ldots t_{n})(a) \defeq \sigma^{\mathbf{A}}\left(\nu^{\mathbf{A}}_{n}(t_{1})(a), \ldots, \nu^{\mathbf{A}}_{n}(t_{n})(a)\right)
\end{displaymath} for $\sigma \in \Sigma_{n}$ and $t_{1},\ldots,t_{n} \in T_{\Sigma}(X_{n})$. We shall denote the image of $\nu^{\mathbf{A}}_{n}$ by $\clone{n}(\mathbf{A})$. The \emph{clone} of the algebra $\mathbf{A}$ is the set \begin{displaymath}
	\Clo (\mathbf{A}) = \bigcup_{n \geq 1} \clone{n}(\mathbf{A}) ,
\end{displaymath} i.e., the set of all finitary operations on $A$ arising from term evaluation with respect to $\mathbf{A}$. Evidently, by jointly extending the maps $\left. \left( \nu^{\mathbf{A}}_{n} \, \right| n \geq 1 \right)$ to the direct sum (a.k.a coproduct) we obtain a surjective mapping \begin{displaymath}
	\nu^{\mathbf{A}} \colon \coprod_{n \geq 1} T_{\Sigma}(X_{n}) \to \Clo (\mathbf{A}) , \quad (n,t) \mapsto \nu^{\mathbf{A}}_{n}(t) .
\end{displaymath} Furthermore, we need to recall some terminology and notation concerning varieties. Given a class $\mathcal{V}$ of algebras of the same type, we denote by $\PA (\mathcal{V})$ the class of all products of algebras from $\mathcal{V}$, by $\SA(\mathcal{V})$ the class of all subalgebras of members of $\mathcal{V}$, and by $\HA(\mathcal{V})$ the class of all homomorphic images of algebras from $\mathcal{V}$. A class $\mathcal{V}$ of algebras being closed with respect to the formation of homomorphic images, subalgebras, and arbitrary products is called \emph{variety}. It is easy to see that $\HSP (\mathcal{V})$ is the smallest variety containing a given class $\mathcal{V}$ of algebras.

Birkhoff's theorem~\cite{birkhoff} implies that an algebra $\mathbf{B}$ belongs to the variety generated by an algebra $\mathbf{A}$ of the same type if and only if $\mathbf{B}$ satisfies all identities (over the common signature) that hold in $\mathbf{A}$. Utilizing the notation introduced above, we can restate this as follows.

\begin{thm}[Birkhoff]\label{theorem:birkhoff} Let $\mathbf{A}$ and $\mathbf{B}$ be two algebras of the same type. Then $\mathbf{B} \in \HSP (\mathbf{A})$ if and only if $\nu^{\mathbf{B}}$ factors through $\nu^{\mathbf{A}}$, i.e., there exists a (necessarily unique and surjective) mapping $\phi \colon \Clo (\mathbf{A}) \to \Clo (\mathbf{B})$ such that $\nu^{\mathbf{B}} = \phi \circ \nu^{\mathbf{A}}$. \end{thm}
	
Concerning algebras $\mathbf{A}$ and $\mathbf{B}$ where $\mathbf{B} \in \HSP (\mathbf{A})$, the unique mapping $\phi \colon \Clo (\mathbf{A}) \to \Clo (\mathbf{B})$ whose existence is asserted by Birkhoff's theorem is called the \emph{natural homomorphism} from $\Clo (\mathbf{A})$ onto $\Clo(\mathbf{B})$. This is due to the fact that $\phi$ preserves the arity of every function, sends projection maps to the corresponding projections onto the same coordinate, and satisfies \begin{displaymath}
	\phi (f(g_{1},\ldots,g_{n})) = \phi(f)(\phi(g_{1}),\ldots,\phi(g_{n}))
\end{displaymath} for any $f \in \clone{n}(\mathbf{A})$ and all $g_{1},\ldots,g_{n} \in \clone{m}(\mathbf{A})$ where $m,n \geq 1$.

In the present paper we shall be concerned with the uniform continuity of natural clone homomorphisms. For this purpose, we need to endow the clone of an algebra with a particular uniform structure. The uniformity under consideration arises in a natural manner as the uniformity of point-wise convergence. We provide a Birkhoff-type characterization of all those members $\mathbf{B}$ of the variety $\HSP (\mathbf{A})$ generated by a given algebra $\mathbf{A}$, for which the natural homomorphism from $\Clo (\mathbf{A})$ to $\Clo (\mathbf{B})$ is uniformly continuous (Theorem~\ref{theorem:uniform.birkhoff} and also Corollary~\ref{corollary:uniform.birkhoff}). This result holds in full generality, i.e., it does not depend on any additional assumptions about the algebras $\mathbf{A}$ and $\mathbf{B}$. Moreover, we show that under the hypothesis of $\mathbf{A}$ being almost locally finite (Definition~\ref{definition:almost.locally.finite}) the considered natural homomorphism is uniformly continuous if and only if it is Cauchy-continuous (Lemma~\ref{lemma:cauchy.continuity}), which immediately yields a Cauchy-continuous version of Birkhoff's theorem (Theorem~\ref{theorem:almost.locally.finite.algebras}). Since the class of almost locally finite algebras encompasses all locally oligomorphic algebras (as well as all locally finite algebras), our results particularly imply a recent theorem by Bodirsky and Pinsker~\cite{TopologicalBirkhoff}.

This paper is organized as follows: in Section~\ref{section:uniform.spaces} we recall some prerequisites concerning uniform spaces, in Section~\ref{section:uniform.birkhoff} we prove the above-mentioned uniformly continuous version of Birkhoff's HSP theorem, and in Section~\ref{section:balanced.group.actions} we establish some general results regarding compactness properties for a uniformly equicontinuous group actions on uniform spaces, which we then apply to almost locally finite algebras in Section~\ref{section:almost.locally.finite.algebras}.


\section{Uniform spaces}\label{section:uniform.spaces}

In this preliminary section we provide some basic concepts concerning uniform spaces. For further reading we refer to \cite{Bourbaki1,Bourbaki2,james}. A \emph{uniformity} on a set $X$ is a filter $\mathcal{U}$ on the set $X \times X$ satisfying the following conditions: \begin{enumerate}
	\item[---] \ Every member of $\mathcal{U}$ contains the diagonal $\Delta_{X} = \{ (x,x) \mid x \in X \}$.
	\item[---] \ If $\alpha$ is a member of $\mathcal{U}$, then so is $\alpha^{-1} = \{ (y,x) \mid (x,y) \in \alpha \}$.
	\item[---] \ For every $\alpha \in \mathcal{U}$, there exists some $\beta \in \mathcal{U}$ such that $\alpha$ contains \begin{displaymath}
					\ \beta \circ \beta = \{ (x,y) \in X \times X \mid \exists z \in X \colon \, \{ (x,z), (z,y) \} \subseteq \beta \} .
				\end{displaymath}
\end{enumerate} By a \emph{uniformity base} on $X$ we mean a filter base on $X \times X$ the filter generated by which is a uniformity on $X$. We refer to the power set of $X \times X$ as the \emph{discrete} uniformity on $X$. A \emph{uniform space} is a non-empty set $X$ equipped with a uniformity on $X$, whose elements are called the \emph{entourages} of~$X$. For a uniform space $X$, the \emph{induced topology} on $X$ is defined as follows: a subset $S \subseteq X$ is \emph{open} in $X$ if for every $x \in S$ there exists an entourage $\alpha$ of~$X$ such that $B_{\alpha}(x) \defeq \{ y \in X \mid (x,y) \in \alpha \}$ is contained in $S$. As usual, the uniform space $X$ is called \emph{precompact} if for every entourage $\alpha$ of $X$ there is a finite subset $F \subseteq X$ such that $X = B_{\alpha}(F) \defeq \bigcup \{ B_{\alpha}(x) \mid x \in F \}$. Recall that \emph{Cauchy filter} on $X$ is a filter $\mathcal{F}$ on the set $X$ such that for every entourage of $\alpha$ of $X$ there exists some $F \in \mathcal{F}$ such that $F \times F \subseteq \alpha$. The uniform space $X$ is called \emph{complete} if every Cauchy filter on $X$ is convergent with respect to the induced topology. In order to also address some matters of continuity, let $Y$ be another uniform space. A map $f \colon X \to Y$ is called \emph{uniformly continuous} if for every entourage $\alpha$ of~$Y$ there exists some entourage $\beta$ of~$X$ such that $(f \times f)(\beta) \subseteq \alpha$. More generally, a set $F \subseteq Y^{X}$ is called \emph{uniformly equicontinuous} if for every entourage $\alpha$ of $Y$ there exists some entourage $\beta$ of~$X$ such that $(f \times f)(\beta) \subseteq \alpha$ whenever $f \in F$. Furthermore, a map $f \colon X \to Y$ is called \emph{Cauchy-continuous} if the image filter \begin{displaymath}
f(\mathcal{F}) \defeq \{ S \subseteq Y \mid \exists F \in \mathcal{F} \colon \, f(F) \subseteq S \}
\end{displaymath} is a Cauchy filter on $Y$ whenever $\mathcal{F}$ is a Cauchy filter on $X$. It is straightforward to check that uniform continuity implies Cauchy-continuity, which in turn implies continuity with regard to the respective topologies.

Let us recall some standard constructions for uniform spaces. Given a uniform space $X$ and any subset $S \subseteq X$, we equip $S$ with the corresponding \emph{subspace uniformity} \begin{displaymath}
	\{ \alpha \cap (S \times S) \mid \alpha \text{ entourage of } X \} ,
\end{displaymath} and we observe that, concerning any other uniform space $Y$, a map $f \colon Y \to S$ is uniformly continuous if and only if $Y \to X, \, y \mapsto f(x)$ is uniformly continuous. Now, let $( X_{i} \mid i \in I )$ be any family of uniform spaces. We endow $\prod_{i \in I}X_{i}$ with the \emph{product uniformity}, i.e., the least uniformity on $\prod_{i \in I}X_{i}$ containing the sets \begin{displaymath}
	\left. \left\{ (x,y) \in \prod_{i \in I}X_{i} \times \prod_{i \in I}X_{i} \, \right| (x_{j},y_{j}) \in \alpha \right\} \qquad (\alpha \text{ entourage of } X_{j}, \, j \in I) ,
\end{displaymath} or equivalently, the least uniformity on $\prod_{i \in I}X_{i}$ so that the map $\pr_{i} \colon \prod_{i \in I}X_{i} \to X_{j}, \, x \mapsto x_{j}$ is uniformly continuous for every $j \in I$. Concerning another uniform space $Y$, it follows that a map $f \colon Y \to \prod_{i \in I}X_{i}$ is uniformly continuous if and only if $Y \to X_{j}, \, y \mapsto f(y)_{j}$ is uniformly continuous for any $j \in I$. Dually, let us equip the direct sum $\coprod_{i \in I} X_{i} = \{ (i,x) \mid i \in I, \, x \in X_{i} \}$ with the usual \emph{coproduct uniformity}, i.e., \begin{displaymath}
	\left. \left\{ \alpha \subseteq \coprod_{i \in I} X_{i} \times \coprod_{i \in I} X_{i} \, \right| \forall j \in I \colon \, \{ (x,y) \in X_{j} \mid ((j,x),(j,y)) \in \alpha \} \text{ is an entrge.~of } X_{j} \right\} ,
\end{displaymath} which is just the largest uniformity on $\coprod_{i \in I} X_{i}$ so that the injection $X_{j} \to \coprod_{i \in I}X_{i}, \, x \mapsto (j,x)$ is uniformly continuous for every $j \in I$. Considering another uniform space $Y$, it is easy to see that any map $f \colon Y \to \prod_{i \in I}X_{i}$ is uniformly continuous if and only if $X_{j} \to Y, \, x \mapsto f(j,x)$ is uniformly continuous for every $j \in I$. Let us make a remark concerning completeness in this context: the product as well as the coproduct of any family of complete uniform spaces is complete. Furthermore, completeness is inherited by closed uniform subspaces.

We conclude this section with a construction concerning non-Hausdorff uniform spaces. In fact, there is a standard procedure for uniform spaces to overcome a lack of the Hausdorff property: if $X$ is a uniform space, then one may consider the equivalence relation \begin{displaymath}
	{\sim} \defeq \bigcap \{ \alpha \mid \alpha \text{ entourage of } X \}
\end{displaymath} and endow the quotient $X_{\ast} \defeq X/{\sim}$ with the uniformity generated by the sets of the form \begin{displaymath}
	\left. \left. \alpha_{\ast} \defeq \left\{ (P,Q) \in X_{\ast} \times X_{\ast} \, \right| P \times Q \subseteq \alpha  \right\} = \left\{ (P,Q) \in X_{\ast} \times X_{\ast} \, \right| (P \times Q) \cap \alpha \ne \emptyset \right\}
\end{displaymath} where $\alpha$ is any entourage of $X$. The resulting uniform space $X_{\ast}$ is a Hausdorff space and therefore called the \emph{Hausdorff quotient} of $X$. It is easy to see that every uniformly continuous map from $X$ to a Hausdorff uniform space factors via a uniformly continuous map through the uniformly continuous map $\pi_{\ast} \colon X \to X_{\ast}, \, x \mapsto [x]_{\sim}$. Furthermore, notice that $\pi_{\ast}^{-1}(\pi_{\ast}(U)) = U$ for every open subset $U \subseteq X$. In particular, $\pi_{\ast}$ is open.

\begin{lem}\label{lemma:dense.compactness} Every uniform space containing a dense compact subset is compact. \end{lem}

\begin{proof} Let $X$ be any uniform space and let $S$ be a dense compact subset of $X$. Consider any open covering $\mathcal{U}$ on $X$. Since $\pi_{\ast}$ is open, $\pi_{\ast}(\mathcal{U})$ is an open covering of $X_{\ast}$. Moreover, $\pi_{\ast}(S)$ is a dense compact subset of the Hausdorff space $X_{\ast}$, and therefore $X_{\ast} = \pi (S)$ is compact. Hence, $\pi_{\ast}(\mathcal{U})$ admits a finite subcover $\mathcal{V}$. Consequently, $\pi_{\ast}^{-1}(\mathcal{V})$ is a finite subcover of $\mathcal{U} = \pi_{\ast}^{-1}(\pi_{\ast}(\mathcal{U}))$, which proves our claim. \end{proof}


\section{A uniform Birkhoff theorem}\label{section:uniform.birkhoff}

In this section we prove a uniformly continuous version of Birkhoff's theorem. For this to make sense, we need to endow the clone of an algebra with a particular uniform structure. The uniformity under consideration arises in a point-wise manner. To explain this, let $\mathbf{A}$ be an algebra. For each $n \geq 1$, regarding the carrier set $A$ as a discrete uniform space, we equip $A^{A^{n}}$ with the corresponding product uniformity and the subset $\clone{n}(\mathbf{A})$ with the induced subspace uniformity. Finally, we endow the disjoint union \begin{displaymath}
	\Clo (\mathbf{A}) = \bigcup_{n \geq 1} \clone{n}(\mathbf{A})
\end{displaymath} with the coproduct uniformity induced by the uniform spaces $( \clone{n}(\mathbf{A}) \mid n \geq 1)$. Of course, this uniformity coincides with the subspace uniformity inherited from the coproduct of the uniform spaces $\left. \left( A^{A^{n}} \, \right| n \geq 1 \right)$, each of which carries the corresponding product uniformity mentioned above. Furthermore, we shall deal with the topological closure of $\Clo (\mathbf{A})$ with respect to point-wise convergence. For each $n \geq 1$, we denote by $\cclone{n}(\mathbf{A})$ the closure of $\clone{n}(\mathbf{A})$ in the topological space $A^{A^{n}}$ and equip it with the subspace uniformity inherited from $A^{A^{n}}$. Moreover, we endow the disjoint union \begin{displaymath}
\cClo (\mathbf{A}) \defeq \bigcup_{n \geq 1} \cclone{n}(\mathbf{A})
\end{displaymath} with the coproduct uniformity induced by the uniform spaces $\left. \left( \cclone{n}(\mathbf{A}) \, \right| n \geq 1 \right)$. Evidently, $\cClo(\mathbf{A})$ is precisely the closure of $\Clo (\mathbf{A})$ in the coproduct of the uniform spaces $\left. \left( A^{A^{n}} \, \right| n \geq 1 \right)$, and its uniform structure agrees with the respective subspace uniformity.

\begin{lem}\label{lemma:varieties.closed.under.directed.colimits} Let $\mathbf{A}$ be an algebra and $\mathcal{V}$ a variety of algebras of the same type as $\mathbf{A}$. Then $\mathbf{A}$ is contained in $\mathcal{V}$ if and only if every finitely generated subalgebra of $\mathbf{A}$ is contained in $\mathcal{V}$. \end{lem}

\begin{proof} Denote by $\mathcal{D}$ the directed set of all finitely generated subalgebras of $\mathbf{A}$. Clearly, if $\mathbf{A}$ is contained in $\mathcal{V}$, then so is any member of $\mathcal{D}$. In order to prove the converse implication, consider the subalgebra $\mathbf{S}$ of $\mathbf{C} \defeq \prod_{\mathbf{B} \in \mathcal{D}} \mathbf{B}$ given on the subset \begin{displaymath}
	S \defeq \{ s \in C \mid \exists \mathbf{B}_{0} \in \mathcal{D} \, \forall \mathbf{B} \in \mathcal{D} \colon \, \mathbf{B}_{0} \leq \mathbf{B} \Longrightarrow s_{\mathbf{B}_{0}} = s_{\mathbf{B}} \} .
\end{displaymath} Define a map $h \colon S \to A$ as follows: for each $s \in S$, let $h(s)$ be the unique element of $A$ with \begin{displaymath}
	\exists \mathbf{B}_{0} \in \mathcal{D} \, \forall \mathbf{B} \in \mathcal{D} \colon \, \mathbf{B}_{0} \leq \mathbf{B} \Longrightarrow h(s) = s_{\mathbf{B}} .
\end{displaymath} It is straightforward to check that $h \colon \mathbf{S} \to \mathbf{A}$ is a surjective homomorphism. Hence, if $\mathcal{D}$ is contained in $\mathcal{V}$, then $\mathbf{A}$ is a member of $\mathcal{V}$ as well. \end{proof}

Furthermore, we need to recall some terminology and notation concerning pseudo-varieties. Given a class $\mathcal{V}$ of algebras of the same type, we denote by $\Pfin (\mathcal{V})$ the class of all finite products of algebras from $\mathcal{V}$. A \emph{pseudo-variety} is a class $\mathcal{V}$ of algebras being closed with respect to homomorphic images, subalgebras, and finite products. It is well known and easy to see that the smallest variety containing a given class $\mathcal{V}$ of algebras is given by $\HSPfin (\mathcal{V})$.

\begin{lem}\label{lemma:uniform.continuity} Let $\mathbf{A}$ and $\mathbf{B}$ be algebras of the same type. If every finitely generated subalgebra of $\mathbf{B}$ is contained in $\HSPfin (\mathbf{A})$, then the natural homomorphism from $\Clo (\mathbf{A})$ to $\Clo (\mathbf{B})$ is uniformly continuous. \end{lem}

\begin{proof} Note that Lemma~\ref{lemma:varieties.closed.under.directed.colimits} asserts that $\mathbf{B}$ is a member of $\HSP (\mathbf{A})$. Consider the natural homomorphism $\phi \colon \Clo (\mathbf{A}) \to \Clo (\mathbf{B})$, whose existence is due to Theorem~\ref{theorem:birkhoff}. Our proof for uniform continuity of $\phi$ proceeds by induction.
	
First assume that $\mathbf{B} \in \Pfin (\mathbf{A})$. Then there is a finite set $S$ such that $\mathbf{B} = \mathbf{A}^{S}$. Let $n \geq 1$. Consider the entourage $\alpha \defeq \{ (f,g) \in \clone{n}(\mathbf{B}) \mid f\vert_{F} = g\vert_{F} \}$ in $\clone{n}(\mathbf{B})$ for some finite subset $F \subseteq B^{n}$. Obviously, $H \defeq \{ b_{j}(s) \mid b \in F , \, s \in S, \, 1 \leq j \leq n \}$ is a finite subset of $A$, and thus $H^{n}$ is a finite subset of $A^{n}$. Hence, $\beta \defeq \{ (f,g) \in \clone{n}(\mathbf{A}) \mid f\vert_{H^{n}} = g\vert_{H^{n}} \}$ is an entourage of $\clone{n}(\mathbf{A})$. We argue that $(\phi \times \phi)(\beta) \subseteq \alpha$. To this end, let $(f,g) \in \beta$. If $b \in F$, then \begin{displaymath}
	\phi(f)(b)(s) = f(b_{1}(s),\ldots,b_{n}(s)) = g(b_{1}(s),\ldots,b_{n}(s)) = \phi(g)(b)(s)
\end{displaymath} for every $s \in S$, and thus $\phi (f)(b) = \phi (g)(b)$. This shows that $(\phi(f),\phi(g)) \in \alpha$. Accordingly, $\phi$ is uniformly continuous.

Now suppose that $\mathbf{B} \in \SA (\mathbf{A})$. Let $n \geq 1$. Consider a finite subset $F \subseteq B^{n}$ and the corresponding entourage $\alpha \defeq \{ (f,g) \in \clone{n}(\mathbf{B}) \mid f\vert_{F} = g\vert_{F} \}$ in $\clone{n}(\mathbf{B})$. Obviously, $F$ is also a finite subset of $A$, and thus $\beta \defeq \{ (f,g) \in \clone{n}(\mathbf{A}) \mid f\vert_{F} = g\vert_{F} \}$ is an entourage of $\clone{n}(\mathbf{A})$. Evidently, $(\phi \times \phi)(\beta) \subseteq \alpha$. Hence, $\phi$ is uniformly continuous.

Next assume that $\mathbf{B} \in \HA (\mathbf{A})$. Then there exists a surjective homomorphism $h \colon \mathbf{A} \to \mathbf{B}$. Let $n \geq 1$. Consider the entourage $\alpha \defeq \{ (f,g) \in \clone{n}(\mathbf{B}) \mid f\vert_{F} = g\vert_{F} \}$ in $\clone{n}(\mathbf{B})$ for some finite subset $F \subseteq B^{n}$. Since the homomorphism $h^{[n]} \colon A^{n} \to B^{n}, \, (a_{1},\ldots,a_{n}) \mapsto (h(a_{1}),\ldots,h(a_{n}))$ is surjective as well, there exists a finite subset $E \subseteq A^{n}$ such that $h^{[n]}(E) = F$. Of course, $\beta \defeq \{ (f,g) \in \clone{n}(\mathbf{A}) \mid f\vert_{E} = g\vert_{E} \}$ is an entourage of $\clone{n}(\mathbf{A})$. We show that $(\phi \times \phi)(\beta) \subseteq \alpha$. For this purpose, let $(f,g) \in \beta$. If $b \in F$, then there exists $a \in E$ with $h^{[n]}(a) = b$, and hence $\phi(f)(b) = h(f(a)) = h(g(a)) = \phi(g)(b)$. This shows that $(\phi(f),\phi(g)) \in \alpha$. Consequently, $\phi$ is uniformly continuous.

Note that the considerations above readily imply that the natural homomorphism from $\Clo (\mathbf{B})$ onto $\Clo (\mathbf{A})$ is uniformly continuous if $\mathbf{B} \in \HSPfin (\mathbf{A})$. Indeed, if $\mathbf{B}$ is a member of $\HSPfin (\mathbf{A})$, then there exist a finite set $S$ and a subalgebra $\mathbf{C} \leq \mathbf{A}^{S}$ such that $\mathbf{B} \in \HA (\mathbf{C})$, and hence the natural homomorphism from $\Clo (\mathbf{A})$ onto $\Clo (\mathbf{B})$ is the composite of the natural homomorphism from $\Clo (\mathbf{A})$ onto $\Clo (\mathbf{A}^{S})$, the one from $\Clo (\mathbf{A}^{S})$ onto $\Clo (\mathbf{C})$, and the one from $\Clo (\mathbf{C})$ onto $\Clo (\mathbf{B})$, and therefore it is uniformly continuous. We are going to use this conclusion subsequently.
	
Finally, we prove the general claim. Suppose that every finitely generated subalgebra of $\mathbf{B}$ is contained in $\HSPfin (\mathbf{A})$. Let $n \geq 1$. Consider the entourage $\alpha \defeq \{ (f,g) \in \clone{n}(\mathbf{B}) \mid f\vert_{F} = g\vert_{F} \}$ in $\clone{n}(\mathbf{B})$ for some finite subset $F \subseteq B^{n}$. Then $S \defeq \{ b_{j} \mid b \in F, \, 1 \leq j \leq n \}$ is finite. Denote by $\mathbf{C}$ the subalgebra of $\mathbf{B}$ generated by $S$. By assumption, $\mathbf{C}$ is contained in $\HSPfin (\mathbf{A})$. Now, $\beta \defeq \{ (f,g) \in \clone{n}(\mathbf{C}) \mid f\vert_{F} = g\vert_{F} \}$ is an entourage in $\clone{n}(\mathbf{C})$. Since the natural homomorphism $\psi \colon \Clo (\mathbf{A}) \to \Clo (\mathbf{C})$ is uniformly continuous, there exists a finite subset $E \subseteq A^{n}$ such that $(\psi \times \psi )(\gamma) \subseteq \beta$ with regard to the entourage $\gamma \defeq \{ (f,g) \in \clone{n}(\mathbf{A}) \mid f\vert_{E} = g\vert_{E} \}$ in $\clone{n}(\mathbf{A})$. We now claim that $(\phi \times \phi)(\gamma) \subseteq \alpha$. To see this, let $(f,g) \in \gamma$. If $b \in F$, then $\phi(f)(b) = \psi(f)(b) = \psi(g)(b) = \phi(g)(b)$. This proves that $(\phi(f),\phi(g)) \in \alpha$. Hence, $\phi$ is uniformly continuous, and we are done. \end{proof}

\begin{thm}\label{theorem:uniform.birkhoff} Let $\mathbf{A}$ and $\mathbf{B}$ be algebras of the same type. The following are equivalent. \begin{enumerate}
	\item[$(1)$] The algebra $\mathbf{B}$ is contained in $\HSP (\mathbf{A})$ and the natural homomorphism from $\Clo (\mathbf{A})$ onto $\Clo (\mathbf{B})$ is uniformly continuous.
	\item[$(2)$] Every finitely generated subalgebra of $\mathbf{B}$ belongs to $\HSPfin (\mathbf{A})$.
\end{enumerate} \end{thm}

\begin{proof} We know that (2) implies (1) by Lemma~\ref{lemma:varieties.closed.under.directed.colimits} and Lemma~\ref{lemma:uniform.continuity}. In order to prove the converse implication, assume that (1) holds. Let $\mathbf{C}$ be a subalgebra of $\mathbf{B}$ generated by a finite sequence of elements $b_{1},\ldots,b_{n} \in B$. Consider the entourage \begin{displaymath}
	\alpha \defeq \{ (f,g) \in \clone{n}(\mathbf{B}) \mid f(b_{1},\ldots,b_{n}) = g(b_{1},\ldots,b_{n}) \}
\end{displaymath} in $\clone{n}(\mathbf{A})$. Since the natural homomorphism $\phi \colon \Clo (\mathbf{A}) \to \Clo (\mathbf{B})$ is uniformly continuous, there exists a finite subset $F \subseteq A^{n}$ such that $(\phi \times \phi)(\beta) \subseteq \alpha$ for $\beta \defeq \{ (f,g) \in \clone{n}(\mathbf{A}) \mid f\vert_{F} = g\vert_{F} \}$. For each $j \in \{ 1,\ldots,n \}$, define an element $d_{j} \in A^{F}$ by setting $d_{j}(a) \defeq a_{j}$ for all $a \in F$. Denote by $\mathbf{D}$ the subalgebra of $\mathbf{A}^{F}$ generated by $\{ d_{1},\ldots,d_{n} \}$. We observe that \begin{displaymath}
	D = \left\{ f(d_{1},\ldots,d_{n}) \left| \, f \in \clone{n}\left(\mathbf{A}^{F}\right) \right\} = \{ (f(a))_{a \in F} \mid f \in \clone{n}(\mathbf{A}) \} . \right.
\end{displaymath} Utilizing this fact, we define a map $h \colon D \to C$ by setting \begin{displaymath}
	h((f(a))_{a \in F}) \defeq \phi (f)(b_{1},\ldots,b_{n}) \qquad (f \in \clone{n}(\mathbf{A})) .
\end{displaymath} Notice that $h$ is well defined: if $f,g \in \clone{n}(\mathbf{A})$ such that $(f(a))_{a \in F} = (g(a))_{a \in F}$, then $(f,g) \in \beta$ and hence $(\phi(f),\phi(g)) \in \alpha$, which means that $\phi(f)(b_{1},\ldots,b_{n}) = \phi(g)(b_{1},\ldots,b_{n})$. Evidently, $h$ is surjective as \begin{displaymath}
	C = \{ g(b_{1},\ldots,b_{n}) \mid g \in \clone{n}(\mathbf{B}) \} = \{ \phi(f)(b_{1},\ldots,b_{n}) \mid f \in \clone{n}(\mathbf{A}) \} = h(D) .
\end{displaymath} We are left to prove that $h \colon \mathbf{D} \to \mathbf{C}$ is a homomorphism. To this end, let $m \geq 0$, consider any $m$-ary operational symbol $\sigma$ of the common signature, and let $c_{1},\ldots,c_{m} \in D$. There exist $f_{1},\ldots,f_{m} \in \clone{n}(\mathbf{A})$ with $c_{\ell} = (f_{\ell}(a))_{a \in F}$ for each $\ell \in \{ 1,\ldots,m \}$. We conclude that \begin{align*}
	\sigma^{\mathbf{C}}(h(c_{1}),\ldots,h(c_{m})) \ &= \ \sigma^{\mathbf{C}}(h((f_{1}(a))_{a \in F}),\ldots,h((f_{m}(a))_{a \in F})) \\
	&= \ \sigma^{\mathbf{B}}(\phi (f_{1})(b_{1},\ldots,b_{n}),\ldots,\phi (f_{m})(b_{1},\ldots,b_{n}))) \\
	&= \ \phi\left(\sigma^{\mathbf{A}}\right)(\phi (f_{1})(b_{1},\ldots,b_{n}),\ldots,\phi (f_{m})(b_{1},\ldots,b_{n}))) \\
	&= \ \phi \left(\sigma^{\mathbf{A}}(f_{1},\ldots,f_{m})\right)(b_{1},\ldots,b_{n}) \\
	&= \ h\left(\psi\left(\sigma^{\mathbf{A}}(f_{1},\ldots,f_{m})\right)(d_{1},\ldots,d_{m})\right) \\
	&= \ h\left(\sigma^{\mathbf{A}^{F}}\left(\psi(f_{1})(d_{1},\ldots,d_{m}),\ldots,\psi(f_{m})(d_{1},\ldots,d_{m})\right)\right) \\
	&= \ h\left(\sigma^{\mathbf{D}}(c_{1},\ldots,c_{m})\right) .
\end{align*} Hence, $h \colon \mathbf{D} \to \mathbf{C}$ is a homomorphism, and thus $\mathbf{C} \in \HSPfin (\mathbf{A})$, which proves our claim. \end{proof}

We conclude this section with a slight reformulation of the theorem above. Towards this aim, we need to recall the concept of a colimit of a directed system of algebras. A \emph{directed system} of algebras is a family of algebras $(\mathbf{A}_{i})_{i \in I}$ of the same type indexed by a directed set $I$ such that $\mathbf{A}_{i} \leq \mathbf{A}_{j}$ for all $i,j \in I$ with $i \leq j$. If $(\mathbf{A}_{i})_{i \in I}$ is a directed system of algebras, then we define its \emph{colimit} to be the algebra $\mathbf{A}$ whose carrier set given by $A \defeq \bigcup\nolimits_{i \in I}A_{i}$ and whose operations are defined by \begin{displaymath}
	\sigma^{\mathbf{A}}(a_{1},\ldots,a_{n}) \defeq \sigma^{\mathbf{A}_{i}}(a_{1},\ldots,a_{n}) \qquad (a_{1},\ldots,a_{n} \in A_{i}, \, i \in I)
\end{displaymath} for any $n$-ary symbol of the common signature and $n \geq 0$. For a class $\mathcal{V}$ of algebras of the same type, we shall denote by $\DA (\mathcal{V})$ the class of all colimits of directed systems of members of $\mathcal{V}$. Note that any algebra arises as the colimit of the directed system of its finitely generated subalgebras. Hence, for an algebra $\mathbf{B}$ and a class $\mathcal{V}$ of algebras of the same type as $\mathbf{A}$ being closed with respect to subalgebras, it is true that $\mathbf{B} \in \DA (\mathcal{V})$ if and only if every finitely generated subalgebra of $\mathbf{B}$ belongs to $\mathcal{V}$. Furthermore, it is not difficult to see that $\DHSPfin (\mathcal{V})$ is the smallest class of algebras containing a given class $\mathcal{V}$ of algebras and being closed with respect to finite products, subalgebras, homomorphic images, and colimits of directed systems of algebras. These observations altogether readily provide us with the following reformulation of Theorem~\ref{theorem:uniform.birkhoff}.

\begin{cor}\label{corollary:uniform.birkhoff} Let $\mathbf{A}$ and $\mathbf{B}$ be algebras of the same type. The following are equivalent. \begin{enumerate}
	\item[$(1)$] The algebra $\mathbf{B}$ is contained in $\HSP (\mathbf{A})$ and the natural homomorphism from $\Clo (\mathbf{A})$ onto $\Clo (\mathbf{B})$ is uniformly continuous.
	\item[$(2)$] $\mathbf{B}$ is a member of $\DHSPfin (\mathbf{A})$.
\end{enumerate} \end{cor}


\section{Balanced group actions}\label{section:balanced.group.actions}

In this section we study \emph{balanced} (i.e., uniformly equicontinuous) group actions on uniform spaces, which is supposed to serve as preparation of Section~\ref{section:almost.locally.finite.algebras}, but which we also regard as being of independent interest. To this end, consider an arbitrary action of a group $G$ on some uniform space $X$. As would seem natural, an entourage $\alpha$ of $X$ is called \emph{$G$-invariant} if $(gx,gy) \in \alpha$ for all $(x,y) \in \alpha$ and $g \in G$. We say that the action of $G$ on $X$ is \emph{balanced} or that $X$ is \emph{$G$-balanced} if the $G$-invariant entourages of $X$ constitute a uniformity base for $X$, which just means that the set $\{ x \mapsto gx \mid g \in G \} \subseteq X^{X}$ is uniformly equicontinuous. Suppose that $X$ is $G$-balanced. We endow the quotient set $X/G \defeq \{ Gx \mid x \in X \}$ with a suitable uniform structure: for any entourage $\alpha$ of $X$, let \begin{align*}
	\alpha /G &\defeq \{ (P,Q) \in (X/G) \times (X/G) \mid (P \times Q) \cap \alpha \ne \emptyset \}  .
\end{align*} It is easy to check that $\{ \alpha /G \mid \alpha \text{ entourage of } X \}$ constitutes a uniformity base on $X/G$, and we equip $X/G$ with the generated uniformity. The following observation clarifies the relation between the introduced uniformity and the usual quotient topology. 

\begin{prop} Consider a balanced action of a group $G$ on a uniform space $X$. Then the map $\pi_{G} \colon X \to X/G, \, x \mapsto Gx$ is continuous and open. In particular, the topology induced by the uniformity on $X/G$ is just the quotient topology generated by $\pi_{G}$. \end{prop}

\begin{proof} Evidently, $\pi_{G}$ is uniformly continuous, because $\alpha \subseteq (\pi_{G} \times \pi_{G})^{-1}(\alpha/G)$ for every entourage $\alpha$ of $X$. Hence, we are left to prove that $\pi_{G}$ is open. So, consider any open subset $U \subseteq X$. We argue that $\pi_{G}(U)$ is open in $X/G$. Let $z \in X$ with $Gz = \pi_{G}(z) \in \pi_{G}(U)$. Then $z = gx$ for some $x \in U$ and $g \in G$. Since $X$ is $G$-balanced and $U$ is open in $X$, there exists a $G$-invariant symmetric entourage $\alpha$ in $X$ such that $B_{\alpha}(x) \subseteq U$. Let $y \in X$ with $\pi_{G}(y) \in B_{\alpha/G}(\pi_{G}(z))$. Due to $\alpha$ being $G$-invariant, it follows that \begin{displaymath}
		y \in B_{\alpha}(Gz) = B_{\alpha}(Gx) = GB_{\alpha}(x) \subseteq GU
\end{displaymath} and hence $\pi_{G}(y) \in \pi_{G}(U)$. This shows that $B_{\alpha/G}(\pi_{G}(z)) \subseteq \pi_{G}(U)$. Thus, $\pi_{G}$ is open. \end{proof}

In the following we shall be interested in conditions asserting compactness of the considered quotient spaces. This is mainly due to the subsequent result, which will be of particular relevance to our considerations in Section~\ref{section:almost.locally.finite.algebras}.

\begin{lem}\label{lemma:automatic.uniform.continuity} Let $G$ be a group, let $X$ and $Y$ be $G$-balanced uniform spaces, and let $f \colon X \to Y$ be a continuous $G$-equivariant map. If $X/G$ is compact, then $f$ is uniformly continuous. \end{lem}

\begin{proof} Suppose $\alpha$ to be a symmetric entourage of $Y$. Then there exists a $G$-invariant symmetric entourage $\beta$ of $Y$ such that $\beta \circ \beta \subseteq \alpha$. As $f$ is continuous, \begin{displaymath}
		\mathcal{U} \defeq \{ \inte (B_{\gamma}(x)) \mid x \in X, \, \gamma \text{ $G$-invariant sym.~entourage of } X, \, f(B_{\gamma \circ \gamma}(x)) \subseteq B_{\beta}(f(x)) \}
\end{displaymath} covers $X$. Since the quotient map $\pi \colon X \to X/G$ is open, $\pi (\mathcal{U})$ is an open cover of $X/G$. By compactness of $X/G$, there exists a finite subset $F \subseteq X$ as well as a family $(\gamma_{z} \mid z \in F)$ of $G$-invariant symmetric entourages of $X$ such that $\{ \pi (B_{\gamma_{z}}(z)) \mid z \in F \}$ covers $X/G$ and $f(B_{\gamma_{z} \circ \gamma_{z}}(z)) \subseteq B_{\beta}(f(z))$ for each $z \in F$. Hence, \begin{displaymath}
	X = \bigcup \{ g(\inte (B_{\gamma_{z}}(z))) \mid g \in G, \, z \in F \} = \bigcup \{ \inte (B_{\gamma_{z}}(gz)) \mid g \in G, \, z \in F \} .
\end{displaymath} Consider the symmetric entourage $\gamma \defeq \bigcap \{ \gamma_{z} \mid z \in F \}$. We claim that $(f \times f)(\gamma) \subseteq \alpha$. To see this, let $(x,y) \in \gamma$. Then there is $z \in F$ as well as $g \in G$ such that $(x,gz) \in \gamma_{z}$. Thus, $(y,gz) \in \gamma_{z} \circ \gamma_{z}$. We observe that \begin{align*}
f(B_{\gamma_{z} \circ \gamma_{z}}(gz)) = f(g(B_{\gamma_{z} \circ \gamma_{z}}(z))) &= g(f(B_{\gamma_{z} \circ \gamma_{z}}(z))) \\
&\subseteq g(B_{\beta}(f(z))) = B_{\beta}(gf(z)) = B_{\beta}(f(gz)) .
\end{align*} Consequently, $(f(x),f(gz)) \in \beta$ and $(f(y),f(gz)) \in \beta$. It follows that $(f(x),f(y)) \in \alpha$. This proves that $f$ is uniformly continuous. \end{proof}

It is well known that any uniform space is compact if and only if it is both precompact and complete. Hence, in order to establish compactness of a certain uniform space, it seems natural to wonder about completeness first. This is the objective of our next result, which will turn out useful in the proof of Proposition~\ref{proposition:countable.almost.locally.finite}. As we shall see in Example~\ref{example:first}, it is not possible remove the countability assumption from the subsequent lemma.

\begin{lem}\label{lemma:complete.quotient} Consider a balanced action of a group $G$ on a complete uniform space $X$. If $X$ admits a countable uniformity base, then $X/G$ is complete. \end{lem}

\begin{proof} Notice that our assumptions imply the existence of a sequence $(\alpha_{n})_{n \in \mathbb{N}}$ of $G$-invariant symmetric entourages of $X$ such that $\{ \alpha_{n} \mid n \in \mathbb{N} \}$ constitutes a uniformity base for $X$ and that ${\alpha_{n+1}} \circ {\alpha_{n+1}} \subseteq {\alpha_{n}}$ for every $n \in \mathbb{N}$. Let $\mathcal{F}$ be a Cauchy filter on $X/G$. For every $n \in \mathbb{N}$, pick an element $y_{n} \in X/G$ such that $F \subseteq B_{\alpha_{n+1}/G}(y_{n})$ for some $F \in \mathcal{F}$. Recursively, choose a sequence $(x_{n})_{n \in \mathbb{N}} \in X^{\mathbb{N}}$ with $\pi_{G}(x_{n}) = y_{n}$ and $(x_{n},x_{n+1}) \in \alpha_{n+1}$ for every $n \in \mathbb{N}$. Hence, \begin{displaymath}
	(x_{n},x_{n+k}) \in {\alpha_{n+1}} \circ \ldots \circ {\alpha_{n+k-1}} \circ {\alpha_{n+k}} \subseteq {\alpha_{n+1}} \circ \ldots \circ {\alpha_{n+k-3}} \circ {\alpha_{n+k-2}} \circ {\alpha_{n+k-2}} \subseteq \ldots \subseteq \alpha_{n}
\end{displaymath} for all $m,k \in \mathbb{N}$. It follows that \begin{displaymath}
	\mathcal{H} \defeq \{ H \subseteq X \mid \exists n \in \mathbb{N} \colon \, \{ x_{m} \mid m \geq n \} \subseteq H \}
\end{displaymath} is a Cauchy filter on $X$. As $X$ is complete, $\mathcal{H}$ converges to some point $x \in X$. We argue that $\mathcal{F}$ converges to $\pi_{G}(x)$ in $X/G$. Let $U$ be a neighborhood of $\pi_{G}(x)$ in $X/G$. Then $\pi_{G}^{-1}(U)$ is a neighborhood of $x$ in $X$, and thus there exists some $n \in \mathbb{N}$ with $B_{\alpha_{n}}(x) \subseteq \pi_{G}^{-1}(U)$. As $(x,x_{n+2}) \in \alpha_{n+1}$ and ${\alpha_{n+1}} \circ {\alpha_{n+1}} \subseteq {\alpha_{n}}$, it follows that $B_{\alpha_{n+1}}(x_{n+2}) \subseteq B_{\alpha_{n}}(x) \subseteq \pi_{G}^{-1}(U)$ and therefore $B_{\alpha_{n+1}/G}(y_{n+2}) \subseteq U$. By choice of $y_{n+2}$, there exists some $F \in \mathcal{F}$ such that \begin{displaymath}
	F \subseteq B_{\alpha_{n+3}/G}(y_{n+2}) \subseteq B_{\alpha_{n+1}/G}(y_{n+2}) \subseteq U .
\end{displaymath} This proves that $\mathcal{F}$ converges in $X/G$. \end{proof}

Before we continue with some more general remarks on balanced group actions, let us discuss a particular class of examples. Of course, any isometric group action on a metric space constitutes a balanced action with respect to the induced uniformity. However, we shall be interested in another source of useful examples. Consider any set $Y$ and a group $G$ acting on another set $X$. Then the canonical action of $G$ on $X^{Y}$ given by \begin{displaymath}
	(gf)(y) \defeq gf(y) \qquad \left( g \in G, \, f \in X^{Y}, \, y \in Y\right)
\end{displaymath} is balanced with respect to the product uniformity $X^{Y}$ induced by the discrete space $X$, since the $G$-invariant entourages of the form \begin{displaymath}
	\left. \ker (\pr_{F}) = \left\{ (f,g) \in X^{Y} \times X^{Y} \, \right| f|_{F} = g|_{F} \right\} \qquad (F \subseteq Y \text{ finite})
\end{displaymath} constitute a uniformity base of $X^{Y}$. Finding tractable conditions implying compactness of $X^{Y}/G$ will be of particular relevance to our considerations in Section~\ref{section:almost.locally.finite.algebras}. Our first observation in this regard concerns the case where $Y$ is countable and is a mere consequence of Lemma~\ref{lemma:complete.quotient}.

\begin{prop}\label{proposition:countable.oligomorphicity} Let $G$ be a group acting on a set $X$. Let $Y$ be a countably infinite set and let $Z$ be a $G$-invariant closed subset of $X^{Y}$. Then $Z/G$ is compact if and only if $\pr_{F}(Z)/G$ is finite for every finite subset $F \subseteq Y$. \end{prop}

\begin{proof} ($\Longrightarrow$) Consider any finite subset $F \subseteq Y$. Clearly, $\pr_{F} \colon Z \to X^{F} , \, f \mapsto f|_{F}$ is continuous with respect to the discrete topology on $X$. Moreover, $\pr_{F}$ is equivariant with respect to the obvious $G$-actions on $Z$ and $X$. Hence, it induces a continuous map \begin{displaymath}
		p \colon Z/G \to X^{F}/G, \quad G f \mapsto G\pr_{F}(f) .
\end{displaymath} Since $Z/G$ is compact, it follows that $\pr_{F}(Z)/G = p(Z/G)$ is compact and hence finite.
	
($\Longleftarrow$) Since $Z$ is closed in the complete space $X^{Y}$, it follows that $Z$ is complete as well. As $Y$ is countable, $Z$ admits a countable uniformity base. Therefore, $Z/G$ is complete due to Lemma~\ref{lemma:complete.quotient}. Hence, it suffices to show that $Z/G$ is precompact. For this purpose, consider the entourage $\alpha \defeq \{ (f,g) \in Z \mid f|_{F} = g|_{F} \}$ in $Z$ with regard to some finite subset $F \subseteq Y$. Since $\pr_{F}(Z)/G$ is finite by assumption, there is a finite subset $E \subseteq Z$ so that $\pr_{F}(Z) = G\pr_{F}(E)$. Of course, this just means that $Z = B_{\alpha}(GE)$, or equivalently, $Z/G = B_{\alpha/G}(E/G)$. \end{proof}

Unfortunately, the previous result cannot be extended to the uncountable setting, as the subsequent example reveals.

\begin{exmpl}\label{example:first} Let $X$ and $Y$ be any two infinite sets such that $\vert X \vert < \vert Y \vert$ and let $G \defeq \Sym (X)$. Then $X^{Y}/G$ is not compact, because \begin{displaymath}
	\left. \left. \left\{ \left\{ f \in X^{Y} \, \right| f(x) = f(y) \right\} \, \right| x,y \in Y, \, x \ne y \right\}
\end{displaymath} is a cover of $X^{Y}$ by $G$-invariant open subsets, which does not admit a finite subcover. It follows that $X^{Y}/H$ is not compact for any subgroup $H$ of $G$. Moreover, $X^{Y}/G$ is not even complete: considering the filter \begin{displaymath}
	\left. \left. \mathcal{F} \defeq \left\{ F \subseteq X^{Y} \right| \, \exists E \subseteq Y \text{ finite}\colon \left\{ f \in X^{Y} \right| \, f\vert_{E} \text{ injective} \right\} \subseteq F \right\}
\end{displaymath} on $X^{Y}$, it is easy to check that the image filter $\pi_{G}(\mathcal{F})$ is a Cauchy filter on $X^{Y}/G$, but that it does not converge in $X^{Y}/G$. \end{exmpl}

\begin{prop}\label{proposition:almost.full.symmetric.group} Let $X$ and $Y$ be infinite sets such that $\vert Y \vert \leq \vert X \vert$. If $G$ is a dense subgroup of $\Sym (X)$ (with regard to the topology of point-wise convergence), then $X^{Y}/G$ is compact. \end{prop}

\begin{proof} Notice that the set $E$ of all equivalence relations on $Y$ constitutes a closed subset of the compact Hausdorff space $\mathscr{P}(Y \times Y) \cong 2^{Y \times Y}$. Thus $E$ is a compact Hausdorff space when endowed with the corresponding subspace topology. Since $\vert Y \vert \leq \vert X \vert$, there exists a map $\phi \colon E \to X^{Y}$ such that $\ker \phi (\theta) = \theta$ for every $\theta \in E$. We argue that $\psi \defeq \pi_{G} \circ \phi \colon E \to X^{Y}/G$ is continuous. So, let $\theta \in E$ and let $U$ be an open neighborhood of $\psi(\theta)$ in $X^{Y}/G$. Then $\pi_{G}^{-1}(U)$ is an open neighborhood of $\phi (\theta)$ in $X^{Y}$, and thus there is a finite subset $F \subseteq Y$ with \begin{displaymath}
	\left. \phi (\theta) \in V \defeq \left\{ f \in X^{Y} \, \right| f|_{F} = \phi(\theta)|_{F} \right\} \subseteq \pi_{G}^{-1}(U) .
\end{displaymath} Evidently, $W \defeq \{ \theta' \in E \mid \theta' \cap (F \times F) = \theta \cap (F \times F) \}$ is an open neighborhood of $\theta$ in $E$. We show that $W \subseteq \psi^{-1}(U)$. Let $\theta' \in W$. Since \begin{displaymath}
	\ker \phi(\theta') \cap (F \times F) = \theta' \cap (F \times F) = \theta \cap (F \times F) = \ker \phi(\theta) \cap (F \times F) 
\end{displaymath} and $G$ is dense in $\Sym (X)$, there exists some element $g \in G$ such that $\phi(\theta')|_{F} = g\phi (\theta)|_{F}$, and therefore $\phi(\theta') = gg^{-1}\phi(\theta') \in gV \subseteq g\pi_{G}^{-1}(U) = \pi_{G}^{-1}(U)$, which proves our claim. Hence, $\psi$ is continuous, and so $\psi (E)$ is compact. Furthermore, $\psi (E)$ is dense in $X^{Y}/G$ because \begin{displaymath}
	f \in \overline{\Sym (X)\phi(\ker f)} = \overline{G\phi(\ker f)}
\end{displaymath} for every $f \in X^{Y}$. Consequently, $X^{Y}/G$ is compact by Lemma~\ref{lemma:dense.compactness}. \end{proof}

The following modification of Example~\ref{example:first} illustrates that the argument of the previous proof cannot be extended to oligomorphic group actions. As usual, we call an action of a group $G$ on a set $X$ \emph{oligomorphic} if $X^{n}/G$ is finite for every $n \geq 1$.

\begin{exmpl}\label{example:second} Let $X$ be any uncountable set and let $\theta$ be an equivalence relation on $X$ such that $X/\theta$ is countably infinite and $\vert [x]_{\theta} \vert = \vert [y]_{\theta} \vert$ for all $x,y \in X$. We observe that $G \defeq \Aut (X,\theta)$ is an oligomorphic subgroup of $\Sym (X)$, because \begin{displaymath}
	x \in Gy \ \Longleftrightarrow \ \forall i,j \in \{ 1,\ldots,n \} \colon \, (x_{i} = x_{j} \Leftrightarrow y_{i} = y_{j}) \wedge ((x_{i}, x_{j}) \in \theta \Leftrightarrow (y_{i},y_{j}) \in \theta)
\end{displaymath} for all $x,y \in X^{n}$ and any $n\geq 1$. However, $X^{X} /G $ is not compact, because \begin{displaymath}
	\left. \left. \left\{ \left\{ f \in X^{X} \, \right| (f(x),f(y)) \in \theta \right\} \, \right| x,y \in X, \, x \ne y \right\}
\end{displaymath} is a cover of $X^{X}$ by $G$-invariant open subsets, which does not admit a finite subcover. It is also clear at which point our argument in the proof of Proposition~\ref{proposition:almost.full.symmetric.group} fails here: the considered map $\phi$ could only be chosen for equivalence relations with index less than or equal to $\vert X \vert$, which is not a closed subset of $E$. \end{exmpl}

Let us also mention a linear version of Proposition~\ref{proposition:almost.full.symmetric.group}.

\begin{prop}\label{proposition:vector.space.automorphisms} Let $V$ and $W$ be infinite dimensional vector spaces over a finite field $\mathbb{F}$ such that $\dim(W) \leq \dim(V)$. If $G$ is a dense subgroup of $\Aut (V)$ (with regard to the topology of point-wise convergence), then $\Hom (W,V)/G$ is compact. \end{prop}

\begin{proof} Note that the set $S$ of all linear subspaces of $W$ forms a closed subset of the compact Hausdorff space $\mathscr{P}(W) \cong 2^{W}$. Thus $S$ is a compact Hausdorff space when endowed with the respective subspace topology. As $\dim(W) \leq \dim(V)$, there is a map $\phi \colon S \to L \defeq \Hom (W,V)$ such that $\Ker \phi (T) = T$ for every $T \in S$. We argue that $\psi \defeq \pi_{G} \circ \phi \colon E \to L/G$ is continuous. So, let $T \in S$ and let $U$ be an open neighborhood of $\psi(T)$ in $L/G$. Then $\pi_{G}^{-1}(U)$ is an open neighborhood of $\phi (T)$ in $L$, and thus there exists a finite subset $F \subseteq W$ such that \begin{displaymath}
	\left. \phi (T) \in V \defeq \left\{ f \in L \, \right| f|_{F} = \phi(\theta)|_{F} \right\} \subseteq \pi_{G}^{-1}(U) .
\end{displaymath} Denote by $Q$ the linear subspace of $W$ generated by $F$. Since the field $\mathbb{F}$ is finite, $Q$ is finite as well and hence $W \defeq \{ T' \in S \mid T' \cap Q = T \cap Q \}$ is an open neighborhood of $T$ in $S$. We show that $W \subseteq \psi^{-1}(U)$. Let $T' \in W$. Since $\Ker \phi(T') \cap Q = T' \cap Q = T \cap Q = \Ker \phi(T) \cap Q$ and $G$ is dense in $\Aut (V)$, there exists some $g \in G$ such that $\phi(T')|_{Q} = g \phi (T)|_{Q}$, and therefore $\phi(T') = gg^{-1}\phi(T') \in gV \subseteq g\pi_{G}^{-1}(U) = \pi_{G}^{-1}(U)$, which proves our claim. Hence, $\psi$ is continuous, and so $\psi (S)$ is compact. Furthermore, $\psi (S)$ is dense in $L/G$ because \begin{displaymath}
	f \in \overline{\Aut (V)\phi(\Ker f)} = \overline{G\phi(\ker f)}
\end{displaymath} for every $f \in L$. Consequently, $L/G$ is compact by Lemma~\ref{lemma:dense.compactness}. \end{proof}

We observe that the considered quotient spaces usually lack the property of being Hausdorff. For instance, this problem occurs in Proposition~\ref{proposition:almost.full.symmetric.group}.

\begin{exmpl}[cf.~Proposition~\ref{proposition:almost.full.symmetric.group}] Let $X$ be an infinite set. Consider the balanced action of the group $G \defeq \Sym (X)$ on the uniform space $X^X$. If $f \colon X \to X$ is injective, but not surjective, then evidently $\id_{X} \notin Gf$, although $(G,Gf) \in \alpha/G$ for every entourage $\alpha$ of $X^X$. Hence, $X^X/G$ is not a Hausdorff space. \end{exmpl}

We conclude this section with an alternative description of the Hausdorff quotient (see Section~\ref{section:uniform.spaces}) for quotients of balanced group actions on uniform spaces, which is nothing but a straightforward generalization of the construction studied in \cite{YaacovTsankov} for isometric group actions. We start with the following observation.

\begin{lem}\label{lemma:closed.partition} Consider a balanced action of a group $G$ on some uniform space $X$. Then $X/\! \! /G \defeq \{ \overline{Gx} \mid x \in X \}$ is a partition of $X$ into closed $G$-invariant subsets. \end{lem}

\begin{proof} Let $x,y \in X$ with $\overline{Gx} \cap \overline{Gy} \ne \emptyset$. We show that $y \in \overline{Gx}$. To this end, let $\alpha$ be a $G$-invariant entourage of $X$. Since $\overline{Gx} \cap \overline{Gy} \ne \emptyset$, we conclude that $(Gy \times Gx) \cap \alpha \ne \emptyset$. Due to $\alpha$ being $G$-invariant, this implies that $(\{ y \} \times Gx) \cap \alpha \ne \emptyset$. Thus, $B_{\alpha}(y) \cap Gx \ne \emptyset$. As $X$ is $G$-balanced, it follows that $y \in \overline{Gx}$. On account of symmetry, we have $x \in \overline{Gy}$ as well, and hence $\overline{Gx} = \overline{Gy}$. This proves our claim. \end{proof}

We endow the quotient suggested by the previous lemma with a suitable uniform structure. Consider a balanced action of a group $G$ on some uniform space $X$. For any entourage $\alpha$ of $X$, let \begin{align*}
	\alpha /\! \! /G &\defeq \{ (P,Q) \in (X/\! \! /G) \times (X/\! \! /G) \mid (P \times Q) \cap \alpha \ne \emptyset \} .
\end{align*} It is straightforward to check that $\{ \alpha /\! \! /G \mid \alpha \text{ entourage of } X \}$ constitutes a uniformity base on the set $X/\! \! /G$, and we equip $X/\! \! /G$ with the generated uniformities. The final observation of this section, whose straightforward proof we leave to the interested reader, clarifies the connection between the introduced quotient spaces.

\begin{prop} Consider a balanced action of a group $G$ on some uniform space $X$. Then \begin{displaymath}
	\phi \colon (X/G)_{\ast} \to X/\! \! /G, \quad [Gx]_{\sim} \mapsto \overline{Gx}
\end{displaymath} is an isomorphism of uniform spaces. In particular, $X/\! \! /G$ is a Hausdorff space. \end{prop}


\section{Almost locally finite algebras}\label{section:almost.locally.finite.algebras}

For a certain class of algebras, which encompasses locally oligomorphic and locally finite algebras, the assumption of uniform continuity in Theorem~\ref{theorem:uniform.birkhoff} can be weakened to Cauchy-continuity (see Theorem~\ref{theorem:almost.locally.finite.algebras}). In order to explain and prove this result, we first need to introduce some additional terminology.

\begin{definition}\label{definition:almost.locally.finite} Let $\mathbf{A}$ be an algebra. Consider the group \begin{displaymath}
\mathscr{G}(\mathbf{A}) \defeq \{ f \in \cclone{1}(\mathbf{A}) \mid \exists g \in \cclone{1}(\mathbf{A}) \colon \, f \circ g = g \circ f = \id_{A} \} .
\end{displaymath} We call $\mathbf{A}$ \emph{almost locally finite} if the quotient space $\cclone{n}(\mathbf{A})/\mathscr{G}(\mathbf{A})$ is compact for any $n \geq 1$. \end{definition}

Recall that an algebra $\mathbf{A}$ is \emph{locally finite} if every of its finitely generated subalgebras is finite. The subsequent result justifies the terminology of Definition~\ref{definition:almost.locally.finite}. For the sake of convenience, let us agree on the following notation: for any finite list of elements $a_{1},\ldots,a_{n}$ of an algebra $\mathbf{A}$, we denote by $\langle a_{1},\ldots,a_{n} \rangle_{\mathbf{A}}$ the subalgebra of $\mathbf{A}$ generated by $\{ a_{1},\ldots,a_{n} \}$.

\begin{prop}\label{proposition:locally.finite.algebras} An algebra $\mathbf{A}$ is locally finite if and only if $\cclone{n}(\mathbf{A})$ is compact for all $n \geq 1$. \end{prop}

\begin{proof} ($\Longrightarrow$) Suppose that $\mathbf{A}$ is locally finite. Let $n \in \mathbb{N}$. We observe that \begin{displaymath}
	\cclone{n}(\mathbf{A}) \, \subseteq \, \prod_{a \in A^{n}} \langle a_{1},\ldots,a_{n} \rangle_{\mathbf{A}} \, \subseteq \, A^{A^{n}} .
\end{displaymath} By Tychonoff's theorem, the space in the middle is compact. Since the first space is closed in the third one, it is also closed in the second one. Hence, the first space is compact itself.
	
($\Longleftarrow$) Let $a_{1},\ldots,a_{n} \in A$. Note that the map $h \colon A^{A^{n}} \to A , \, f \mapsto f(a_{1},\ldots,a_{n})$ is continuous with respect to the discrete topology on $A$ and the corresponding product topology on $A^{A^{n}}$. Besides, $h\left(\cclone{n}(\mathbf{A})\right) = h(\clone{n}(\mathbf{A})) = \langle a_{1},\ldots,a_{n} \rangle_{\mathbf{A}}$. Since $\cclone{n}(\mathbf{A})$ is compact, we conclude that $\langle a_{1},\ldots,a_{n} \rangle_{\mathbf{A}}$ is compact and hence finite. \end{proof}

\begin{cor} Every locally finite algebra is almost locally finite. \end{cor}

Let us point out another class of almost locally finite algebras.

\begin{prop} If $\mathbf{A}$ is an algebra and $\mathscr{G}(\mathbf{A}) = \Sym (A)$, then $\mathbf{A}$ is almost locally finite. \end{prop}

\begin{proof} Without loss of generality, we may assume that $A$ is infinite. Let $\mathscr{G} = \mathscr{G}(\mathbf{A}) = \Sym (A)$ and consider any $n \geq 1$. By assumption, $\vert A^{n} \vert = \vert A \vert$ and therefore $A^{A^{n}}/\mathscr{G}$ is compact by Proposition~\ref{proposition:almost.full.symmetric.group}. Since $\cclone{n}(\mathbf{A})$ is a $\mathscr{G}$-invariant closed subset of $A^{A^{n}}$, it follows that $\cclone{n}(\mathbf{A})/\mathscr{G}$ is closed in $A^{A^{n}}/\mathscr{G}$, and thus $\cclone{n}(\mathbf{A})/\mathscr{G}$ is compact itself. \end{proof}

The next result, which follows rather immediately from Proposition~\ref{proposition:countable.oligomorphicity}, is supposed to shed some more light on the concept of almost locally finiteness for countable algebras.

\begin{prop}\label{proposition:countable.almost.locally.finite} Let $\mathbf{A}$ be a countable algebra and $\mathscr{G} \defeq \mathscr{G}(\mathbf{A})$. Then $\mathbf{A}$ is almost locally finite if and only if $\mathbf{B}/\mathscr{G}$ is finite for every finitely generated subalgebra $\mathbf{B} \leq \mathbf{A}^{n}$ with $n \geq 1$. \end{prop}

\begin{proof} ($\Longrightarrow$) Let $n \geq 1$ and $a_{1},\ldots,a_{m} \in A^{n}$. Denote by $\mathbf{B}$ the subalgebra of $\mathbf{A}^{n}$ generated by $\{ a_{1},\ldots,a_{m} \}$. Consider the map $p \colon \cclone{m}(\mathbf{A}) \to A^{n}$ given by \begin{displaymath} 
	p(f)(j) \defeq f(a_{1}(j),\ldots,a_{m}(j)) \qquad \left(f \in \cclone{n}(\mathbf{A}), \, j \in \{ 1,\ldots, n \} \right) .
\end{displaymath} We observe that $\langle a_{1},\ldots,a_{m} \rangle_{\mathbf{A}^{n}} = p(\clone{m}(\mathbf{A})) = p(\cclone{m}(\mathbf{A}))$. Since $\cclone{m}(\mathbf{A})/\mathscr{G}$ is compact, this readily implies that $\langle a_{1},\ldots,a_{m} \rangle_{\mathbf{A}^{n}}/\mathscr{G} = p(\cclone{m}(\mathbf{A}))/\mathscr{G}$ is finite due to Proposition~\ref{proposition:countable.oligomorphicity}.
	
($\Longleftarrow$)  Let $n \geq 1$. We utilize Proposition~\ref{proposition:countable.oligomorphicity}, according to which it suffices to show that $\pr_{F}(\cclone{n}(\mathbf{A}))/\mathscr{G}$ is finite for every finite subset $F \subseteq A^{n}$. So, let $F \subseteq A^{n}$ be finite. For each $j \in \{ 1,\ldots,n \}$, define an element $b_{j} \in A^{F}$ by setting $b_{j}(a) \defeq a_{j}$ whenever $a \in F$. Denote by $\mathbf{B}$ the subalgebra of $\mathbf{A}^{F}$ generated by $\{ b_{1},\ldots,b_{n} \}$. Of course, \begin{displaymath}
	B = \left\{ f(b_{1},\ldots,b_{n}) \left| \, f \in \clone{n}\left(\mathbf{A}^{F}\right) \right\} = \pr_{F}\left( \clone{n}(\mathbf{A}) \right) = \pr_{F}\left( \cclone{n}(\mathbf{A}) \right) . \right.
\end{displaymath} Hence, $\pr_{F}\left( \cclone{n}(\mathbf{A}) \right) = B/\mathscr{G}$ is finite by assumption. This completes the proof. \end{proof}

Let us reformulate the result above in terms of oligomorphicity. As proposed by Bodirsky and Pinsker~\cite{TopologicalBirkhoff}, we call an algebra $\mathbf{A}$ \emph{locally oligomorphic} if $A$ is countable and the action of $\mathscr{G}(\mathbf{A})$ on $A$ is oligomorphic.

\begin{prop} Let $\mathbf{A}$ be a countable algebra and $\mathscr{G} \defeq \mathscr{G}(\mathbf{A})$. Then $\mathbf{A}$ is almost locally finite if and only if the action of $\mathscr{G}$ on every finitely generated subalgebra of $\mathbf{A}$ is oligomorphic. \end{prop}

\begin{proof} ($\Longrightarrow$) Consider any finitely generated subalgebra $\mathbf{B} \leq \mathbf{A}$. For every $n \geq 1$, it follows that $\mathbf{B}^{n}$ is a finitely generated subalgebra of $\mathbf{A}^{n}$, whence $\mathbf{B}^{n}/\mathscr{G}$ is finite by Proposition~\ref{proposition:countable.almost.locally.finite}. This shows that $\mathscr{G}$ acts oligomorphically on $\mathbf{B}$.
	
($\Longleftarrow$) We utilize Proposition~\ref{proposition:countable.almost.locally.finite}. Let $n \geq 1$ and consider a finitely generated subalgebra $\mathbf{B} \leq \mathbf{A}^{n}$. Then there exists some finitely generated subalgebra $\mathbf{C} \leq \mathbf{A}$ such that $\mathbf{B} \leq \mathbf{C}^{n}$. By assumption, the action of $\mathscr{G}$ on $\mathbf{C}$ is oligomorphic, and thus $\mathbf{C}^{n}/\mathscr{G}$ is finite, which readily implies that $\mathbf{B}/\mathscr{G}$, too. By Proposition~\ref{proposition:countable.almost.locally.finite}, it follows that $\mathbf{A}$ is almost locally finite. \end{proof}

\begin{cor}[cf.~\cite{TopologicalBirkhoff}] Every locally oligomorphic algebra is almost locally finite. \end{cor}

The observation above is interesting insofar as the only algebras being both locally finite and locally oligomorphic are the finite ones -- which is why one might think of the class of locally finite algebras as being orthogonal to the class of locally oligomorphic algebras. However, both classes are contained in the class of almost locally finite algebras and hence will be covered by Lemma~\ref{lemma:cauchy.continuity} and Theorem~\ref{theorem:almost.locally.finite.algebras}.

\begin{lem}\label{lemma:cauchy.continuity} Let $\mathbf{A}$ be an almost locally finite algebra and let $\mathbf{B} \in \HSP (\mathbf{A})$. If the natural homomorphism from $\Clo (\mathbf{A})$ to $\Clo (\mathbf{B})$ is Cauchy-continuous, then it is uniformly continuous. \end{lem}

\begin{proof} Let $n \geq 1$. Suppose that the natural homomorphism $\phi \colon \clone{n}(\mathbf{A}) \to \clone{n}(\mathbf{B})$ is Cauchy-continuous. As $\cclone{n}(\mathbf{B})$ is a complete Hausdorff uniform space, an elementary theorem from set-theoretic topology asserts that $\phi$ admits a unique continuous extension $\psi \colon \cclone{n}(\mathbf{A}) \to \cclone{n}(\mathbf{B})$. Since both the action of $\mathscr{G} \defeq \mathscr{G}(\mathbf{A})$ on $\cclone{n}(\mathbf{A})$ and the one on $\cclone{n}(\mathbf{B})$ are balanced and $\psi \colon \cclone{n}(\mathbf{A}) \to \cclone{n}(\mathbf{B})$ is $\mathscr{G}$-equivariant, Lemma~\ref{lemma:automatic.uniform.continuity} states that $\psi$ is uniformly continuous. Hence, $\phi$ is uniformly continuous as well. \end{proof}

Combining the lemma above with Theorem~\ref{theorem:uniform.birkhoff} and Corollary~\ref{corollary:uniform.birkhoff}, we immediately obtain our final result concerning almost locally finite algebras.

\begin{thm}\label{theorem:almost.locally.finite.algebras} Let $\mathbf{A}$ and $\mathbf{B}$ be two algebras of the same type. If $\mathbf{A}$ is almost locally finite, then the following are equivalent. \begin{enumerate}
	\item[$(1)$] The algebra $\mathbf{B}$ is contained in $\HSP (\mathbf{A})$ and the natural homomorphism from $\Clo (\mathbf{A})$ onto $\Clo (\mathbf{B})$ is uniformly continuous.
	\item[$(2)$] The algebra $\mathbf{B}$ is contained in $\HSP (\mathbf{A})$ and the natural homomorphism from $\Clo (\mathbf{A})$ onto $\Clo (\mathbf{B})$ is Cauchy-continuous.
	\item[$(3)$] Every finitely generated subalgebra of $\mathbf{B}$ belongs to $\HSPfin (\mathbf{A})$.
	\item[$(4)$] $\mathbf{B}$ is a member of $\DHSPfin (\mathbf{A})$.
\end{enumerate} \end{thm}

\section*{Acknowledgments}

The author is supported by funding of the Excellence Initiative by the German Federal and State Governments. Furthermore, the author wants to express his gratitude towards Manuel Bodirsky and Marcello Mamino for several interesting discussions -- in particular towards Marcello Mamino for pointing out Examples~\ref{example:first} and~\ref{example:second} to the author.

\printbibliography

\end{document}